\documentclass{amsart}

\newif\ifdraft\draftfalse
\newif\ifcite\citefalse
\newif\ifblow\blowfalse

\usepackage{amsfonts,amssymb, amsmath}
\usepackage{oldgerm}
\usepackage{amscd}
\usepackage{mathdots}
\usepackage{mathrsfs} 
\usepackage[hyperindex]{hyperref}  
\usepackage[alphabetic]{amsrefs}
\usepackage{bm}
\ifdraft\ifcite\usepackage{showkeys}\else\usepackage[notcite,notref]{showkeys}\fi\fi

\usepackage{xypic}
\newdir{ >}{{}*!/-5pt/\dir{>}}

\newtheorem{proposition}[equation]{Proposition}
\newtheorem{theorem}[equation]{Theorem}

\newtheorem{lemma}[equation]{Lemma}

\theoremstyle{remark}

\theoremstyle{definition}
\newtheorem{definition}[equation]{Definition}

\theoremstyle{remark}

\newtheorem{remark}[equation]{Remark}

\newtheorem{rem}[equation]{Remark}

\numberwithin{equation}{section}

\def\bc{\begin{cases}}
\def\ec{\end{cases}}

\def\ol{\overline}
\def\ul{\underline}

\def\a{\alpha}
\def\Bbb{\mathbb}

\def\t{\tilde}

\def\ca{{\mathcal A}}

\def\cj{{\mathcal J}}

\def\cl{{\mathcal L}}

\def\bc{{\mathbb C}}

\def\bn{{\mathbb N}}

\def\br{{\mathbb R}}

\def\bz{{\mathbb Z}}

\def\er{\eqref}

\def\bz{\mathbb Z}

\def\br{\mathbb R}
\def\bc{\mathbb C}

\def\lp2{L_pH_{2p}}
\def\bean{\begin{eqnarray}}
\def\eean{\end{eqnarray}}
\def\bea{\begin{eqnarray*}}
\def\eea{\end{eqnarray*}}
\def\beq{\begin{equation}}
\def\eeq{\end{equation}}
\def\beq*{\begin{equation*}}
\def\eeq*{\end{equation*}}
\def\bal{\begin{align*}}
\def\eal{\end{align*}}
\def\baln{\begin{align}}
\def\ealn{\end{align}}

\def\beg{\begin{gather*}}
\def\eng{\end{gather*}}
\def\bqu{\begin{question}}
\def\equ{\end{question}}

\def\ban{\begin{proof}[Answer]}
\def\ean{\end{proof}}

\def\p{\partial}

\def\on{\operatorname}

\def\bqu{\begin{question}}
\def\equ{\end{question}}

\def\0110{\begin{matrix} 0 & 1\\1&0\end{matrix}}

\def\t{\tilde}

\def\fg{\mathfrak{g}}
\def\fh{\mathfrak{h}}

\def\fl{\mathfrak{l}}

\def\fn{\mathfrak{n}}
\def\fo{\mathfrak{o}}
\def\fp{\mathfrak{p}}

\def\fs{\mathfrak{s}}

\def\ban{\begin{proof}[Answer]}
\def\ean{\end{proof}}

\def\ben{\begin{equation}}
\def\een{\end{equation}}

\def\j1{{(j+1)}}

\newcommand\ve{\varepsilon}
\newcommand{\hk}{\Big\lfloor\frac{k}{2}\Big\rfloor}
\newcommand{\bnu}{\boldsymbol \nu}

\begin{document}

\title[Toda systems of types $C$ and $B$ with singular sources]{Classification of solutions to Toda systems of types $C$ and $B$ with singular sources}

\author{Zhaohu Nie}
\email{zhaohu.nie@usu.edu}

\address{Department of Mathematics and Statistics, Utah State University, Logan, UT 84322-3900}



\begin{abstract}
In this paper, the classification in \cite{LWY} of solutions to Toda systems of type $A$ with singular sources is generalized to Toda systems of types $C$ and $B$. 
Like in the $A$ case, the solution space is shown to be parametrized by the abelian subgroup and a subgroup of the unipotent subgroup 
in the Iwasawa decomposition of the corresponding complex simple Lie group.
The method is by studying the Toda systems of types $C$ and $B$ as reductions of Toda systems of type $A$ with symmetries. 
The theories of Toda systems as integrable systems as developed in \cite{L, LS-book, N1, N2}, in particular the $W$-symmetries and the iterated integral solutions, play essential roles in this work, together with 
certain characterizing properties of minors of symplectic and orthogonal matrices. 
\end{abstract}


\maketitle

\section{Introduction}

The Toda systems on the plane with singular sources that we will consider in this paper are the elliptic versions of the conformal Toda field theories. Toda filed theories are fundamental examples of integrable systems and have many applications in mathematics and physics. Toda field theories have been widely studied in the literature using various viewpoints and techniques, and we refer the reader to a very partial list of \cite{L, LS-book, Feher, N1, N2}. The elliptic Toda systems are closely related to the non-relativistic and relativistic non-abelian Chern-Simons gauge field theory \cite{Yang}. In particular, the results of this paper will be useful in constructing non-topological solutions to the relativistic Chern-Simons gauge field theory, following the works of \cite{CI, ALW} where the cases of $A_1$, $A_2$ and $C_2$ are treated. 

Let us introduce the systems that we will consider following the fundamental work \cite{LWY}. To each complex simple Lie algebra there is associated a Toda system. A complex simple Lie algebra $\fg$ of rank $n$ is classified by its Cartan matrix $(a_{ij})$, a rank-$n$ matrix with integer entries such that $a_{ii}=2$ and $a_{ij} \leq 0$ when $i\neq j$. 
The classification result of Killing and Cartan states that the complex simple Lie algebras come in four classical series called $A_n$, $B_n$, $C_n$ and $D_n$ with the respective simple Lie algebras being $\fs\fl_{n+1}$, $\fs\fo_{2n+1}$, $\fs\fp_{2n}$, $\fs\fo_{2n}$, together with five exceptional Lie algebras $G_2, F_4, E_{6,7,8}$. We refer the reader to, for example, \cite{FH,Knapp} for basic Lie theory. 

To each simple Lie algebra $\fg$ of rank $n$ with Cartan matrix $(a_{ij})_{i, j=1}^n$, we consider 
the following associated Toda system on the plane with singular sources at the origin and with finite integrals
\begin{equation}
\label{toda}
\begin{cases}
\displaystyle
\Delta u_i + 4\sum_{j=1}^n a_{ij} e^{u_j} = 4\pi \gamma_i \delta_0, & \gamma_i>-1,\\
\displaystyle
\int_{\br^2} e^{u_i} \,dx< \infty, & 1\leq i\leq n.
\end{cases}
\end{equation}
Here the $u_i$ are real-valued functions on the plane with coordinates $x = (x_1, x_2)$, $\Delta = \frac{\p^2}{\p x_1^2} + \frac{\p^2}{\p x_2^2}$, and the $\delta_0$ is the Dirac delta function at the origin. 
The work \cite{LWY} is the first to systematically study Toda systems with singular sources at the origin and with the finite integral conditions. Such a setup leads to the powerful classification result in \cite{LWY} and is useful in studying the non-topological solutions of the relativistic Chern-Simons gauge theory in \cite{ALW}. 

In solving \er{toda}, we will heavily use the complex coordinates $z=x_1+i\, x_2$ and $\bar z=x_1-i\, x_2$. For simplicity, we write $\p_z = \frac{\p}{\p z}= \frac{1}{2}(\frac{\p}{\p x_1} - i \frac{\p}{\p x_2})$, and similarly $\p_{\bar z} = \frac{\p}{\p \bar z}= \frac{1}{2}(\frac{\p}{\p x_1} + i \frac{\p}{\p x_2})$. The Laplace operator is then 
\begin{equation}\label{lap}
\Delta = 4\partial_z \partial_{\bar z}.
\end{equation}
The coefficient $4$ here is responsible for the slightly unconventional coefficient 4 on the left of \er{toda}. For more on this, see Remark \ref{coeff 4} at the end of this introduction. 


The local solutions of the systems \er{toda}, when viewed using \er{lap}, have been well studied in terms of integrable systems especially by the works of Leznov and Saveliev \cite{L, LS, LS-book}. Given $n$ arbitrary functions $\phi_i(z)$ of $z$ and $n$ arbitrary functions $\psi_i(\bar z)$ of $\bar z$, local complexed-valued solutions $u_1,\dots, u_n$ can be constructed using some representation theory of the simple Lie group $G$ corresponding to the Cartan matrix $(a_{ij})$. The requirements of global real solutions with finite integrals in \er{toda} will put on rigid restrictions on these arbitrary functions, and the difficulty of the classification lies in proving that all the solutions to \er{toda} are of certain forms.  

In the literature, general classification results of the global solutions to the Toda systems \er{toda} with finite integrals only exists for Lie algebras of type $A$, although \cite{ALW} and \cite{ALW2} have treated the $C_2$ and $G_2$ cases. Such classification results started without the singular sources, that is, when $\gamma_i = 0$. When the Lie algebra is $A_1 = \fs\fl_2$, the system \er{toda} becomes the ubiquitous Liouville equation, and Chen and Li \cite{CL} classified their finite-integral global solutions by the moving plane method following \cite{GNN}. Jost and Wang \cite{JW} extended the classification to the general $A_n$ cases without singular sources using some theorem from algebraic geometry about totally unramified curves. In the thorough work \cite{LWY}, Lin, Wei and Ye established the complete classification for general $A_n$ systems with singular sources together with the non-degeneracy properties for such solutions. 
One ingenious step in their proofs is to control the solutions by an ODE using the $W$-symmetries of the Toda systems, which are shown to take simple forms in the current setup. 
This author has studied the $W$-symmetries in \cite{N2} under the name of characteristic integrals. 

In this paper, we obtain the complete classification of the solutions to the system \er{toda} for Lie algebras of types $C$ and $B$. Our convention is that the Cartan matrices for $A_n$, $C_n$ and $B_n$ are respectively 
\begin{equation}
\label{Cartan-A}
A_n : \begin{pmatrix}
2 & -1 & & &\\
-1 & 2 & -1 & & \\
 & \ddots & \ddots & \ddots& \\
 & &-1 & 2 & -1\\
 & & & -1 & 2
 \end{pmatrix},
 \end{equation}
 \begin{equation}
 \label{Cartan-C}
 C_n : \begin{pmatrix}
2 & -1 & & &\\
-1 & 2 & -1 & & \\
 & \ddots & \ddots & \ddots& \\
 & &-1 & 2 & -1\\
 & & & -2 & 2
 \end{pmatrix},
 \end{equation}
 \begin{equation}
\label{Cartan-B}
B_n : \begin{pmatrix}
2 & -1 & & &\\
-1 & 2 & -1 & & \\
 & \ddots & \ddots & \ddots& \\
 & &-1 & 2 & -2\\
 & & & -1 & 2
 \end{pmatrix}.
 \end{equation}

First we cast the classification result of \cite{LWY} in the form that we would like to generalize, and we refer the reader to \cite{Knapp, Helgason} for the Iwasawa decomposition. 
\begin{theorem}[\cite{LWY}]\label{A in intro} For the Lie algebra $A_n$, the corresponding simple complex Lie group is $G = SL(n+1, \bc)$. Let $G=KAN$ be its Iwasawa decomposition, where $K=SU(n+1)$ is compact, $A = \br_+ ^n$ is abelian, and $N$ is the unipotent subgroup of unipotent lower-triangular matrices. 
The space of solutions to \er{toda} of type $A_n$ is parametrized by $AN_\Gamma$, where $N_\Gamma$ is a subgroup of $N$ determined by the set of $\gamma_i$. In particular, if all the $\gamma_i$ are integers, then $N_\Gamma  =N$ and the dimension of the solution space is the dimension of the space $AN$, which is $n^2+2n$. Furthermore there are concrete formulas for the solution associated to an element in the relevant space $AN_\Gamma$. 
\end{theorem}

Our main theorem is that similar results continue to hold in the $C_n$ and $B_n$ cases, and it follows from Section \ref{red} and Propositions \ref{main}, \ref{make them}, and \ref{NGm}. 
\begin{theorem}\label{main intro}
\quad {\rm (i)} \quad
For the Lie algebra $C_n$, the corresponding simple complex Lie group is $G = Sp(2n, \bc)$, the group of symplectic matrices preserving $J_{2n}$ in \er{def j}.  
Let $G=KAN$ be its Iwasawa decomposition, where $K=Sp(2n)$ is compact, $A = \br_+ ^n$ is abelian, and $N$ is the unipotent subgroup of unipotent lower-triangular matrices in $Sp(2n, \bc)$. 
The space of solutions to \er{toda} of type $C_n$ is parametrized by $AN_\Gamma$, where $N_\Gamma$ is a subgroup of $N$ determined by the set of $\gamma_i$. In particular, if all the $\gamma_i$ are integers, then $N_\Gamma  =N$ and the dimension of the solution space is the dimension of  the space $AN$, which is $2n^2+n$. Furthermore there are concrete formulas for the solution associated to an element in the relevant space $AN_\Gamma$. 

\medskip
{\rm (ii)}\quad 
For the Lie algebra $B_n$, the corresponding simple complex Lie group is $G = SO(2n+1, \bc)$, the group of special orthogonal matrices preserving $J_{2n+1}$ in \er{def j}.  
Let $G=KAN$ be its Iwasawa decomposition, where $K=SO(2n+1)$ is compact, $A = \br_+ ^n$ is abelian, and $N$ is the unipotent subgroup of unipotent lower-triangular matrices in $SO(2n+1, \bc)$. 
The space of solutions to \er{toda} of type $B_n$ is parametrized by $AN_\Gamma$, where $N_\Gamma$ is a subgroup of $N$ determined by the set of $\gamma_i$. In particular, if all the $\gamma_i$ are integers, then $N_\Gamma  =N$ and the dimension of the solution space is the dimension of  the space $AN$, which is $2n^2+n$. Furthermore there are concrete formulas for the solution associated to an element in the relevant space $AN_\Gamma$. 
\end{theorem}

The system \er{toda} has another version which is easier to use for many purposes. 
Let $u_i = \sum_{j=1}^n a_{ij} U_j$. Then the $U_i$ satisfy 
\begin{equation}
\label{toda2}
\begin{cases}
\displaystyle
\Delta U_{i} + 4\exp\Big(\sum_{j=1}^n  a_{ij} U_j\Big) = 4\pi \alpha_i \delta_0\\
\displaystyle
\int_{\br^2} e^{\sum_j a_{ij}U_j} \, dx< \infty, \qquad 1\leq i\leq n.
\end{cases}
\end{equation}
Here 
\begin{equation}\label{def alpha}
\alpha_i = \sum_{j=1}^n a^{ij} \gamma_j,
\end{equation}
where $(a^{ij})$ is the inverse matrix of $(a_{ij})$. 
By \er{lap}, the first equation in \er{toda2} is the same as 
\begin{equation}\label{U in zz}
U_{i,z \bar z} + \exp\Big(\sum_{j=1}^n a_{ij} U_j\Big) = \pi \alpha_i \delta_0,
\end{equation}
and it is this form that is usually called Toda field theory. 

We emphasize that the main tools in this paper come from the theory of the Toda systems as integrable systems as developed by Leznov and Saveliev \cite{L, LS, LS-book}, with some further development and clarification by this author \cite{N1, N2}. Using the structure of $W$-symmetries in \cite{N2}, we can derive some of the results in \cite{LWY} concerning Toda systems of type $A$ more conceptually and quickly. Furthermore the results in \cite{LWY} are best presented using the iterated integral viewpoint explicitly spelled out in \cite{L}. 
The results in Theorem \ref{main intro} are obtained using the structure theory of the solutions to Toda systems of types $C$ and $B$ in \cite{N1}. It is a well-known fact that the Toda systems of type $C_n$ and $B_n$ are reductions of those of types $A_{2n-1}$ and $A_{2n}$. Therefore the results in \cite{LWY} for type $A$ lead to results for type $C_n$ and $B_n$. However without the correct viewpoint, this reduction procedure becomes tedious and un-illuminating. In the literature, \cite{ALW,ALW2} worked out such reductions for the cases of $C_2$ and $G_2$ with long formulas. In this paper, we show that with the correct viewpoint, the results are still expressed using the nice structure of Lie groups. 

This paper is organized as follows. 
In Section \ref{revisit}, we clarify some key results of \cite{LWY} using the structure theory in \cite{N2}, and cast the results in \cite{LWY} in the framework using iterated integrals \cite{L}. In Section \ref{minors}, we establish some characterizing algebraic properties of minors of symplectic and orthogonal matrices. Here and throughout the paper, we are able to treat the symplectic and the orthogonal cases on the same footing by the bilinear form in \er{def j}. In Section \ref{red}, we spell out the $C_n$ and $B_n$ systems as reductions of $A_k$ systems with symmetries. Then finally in Section \ref{symmetry}, we show that the symmetry reduction requirement forces the group elements in the solutions to the $A_k$ system to be more special, that is, they are symplectic or orthogonal respectively. Finally in Section \ref{further}, we study the subgroup $N_\Gamma$ of $N$ determined by the $\gamma_i$ for the solutiosn to be well-defined on $\bc^*=\bc\backslash \{0\}$. In the appendices, we show the examples of $C_3$ and $B_2$ Toda systems to demonstrate our results. 

\begin{remark} Toda systems \er{toda} for Lie algebras of types $D, E$ and $F$ can't be studied along the line of this paper since they are not reductions of the type $A$. Their study will be pursued in a future work. 
\end{remark}

\begin{remark}\label{coeff 4} The coefficients $4$ on the left hand sides of \er{toda} and \er{toda2} come from \er{lap}. 
With this coefficient 4, we get the most convenient form of equation \er{U in zz} to which we will apply many tools from Toda field theories. This coefficient can be easily dealt with as follows. 

The solutions $u_i$ to \er{toda} and the solutions $\hat u_i$ to the more conventional equation in \cite{LWY},
$
\Delta \hat u_i + \sum_{j=1}^n a_{ij} e^{\hat u_j} = 4\pi \gamma_i \delta_0,
$
are related by 
$$
\hat u_i = u_i + \ln 4, \quad 1\leq i\leq n.
$$
Similarly the solutions $U_i$ to \er{toda2} and the solutions $\hat U_i$ to the more conventional equation in \cite{LWY},
$
\Delta \hat U_{i} + \exp\Big(\sum_{j=1}^n  a_{ij} \hat U_j\Big) = 4\pi \alpha_i \delta_0,
$
are related by 
$$
\hat U_i = U_i + (\ln 4) \sum_{j=1}^n a^{ij}, \quad 1\leq i\leq n.
$$

Our current choice of coefficients makes many formulas easier. For example, the coefficient $2^{-n(n+1)}$ in Eq. (1.11) of \cite{LWY} would be gone under our convention. 
\end{remark}

\medskip
\noindent{\bf Acknowledgment.} The author thanks Prof. Z.-Q. Wang for bringing his attention to the work of \cite{LWY} and the Chern Institute of Mathematics at Nankai University for a pleasant visit in May 2014 where this work was started. He also thanks Ming Xu for several helpful discussions. 

\section{The $A$ case revisited using Toda field theories}\label{revisit}

In this section, we revisit the fundamental classification result in \cite{LWY} for solutions to Toda systems \er{toda} with singular sources for Lie algebras of type $A$, that is, with Cartan matrix \er{Cartan-A}. We apply the theory of Toda systems as integrable systems to reformulate some of their proofs and results. 

For the reader's convenience, first we recall the classification result in \cite{LWY} for type $A$. For Toda systems \er{toda2} of types $A$, $C$, $B$, and $G_2$, one simplification is that $U_1$ decides all the other $U_i$ by the shapes of their Cartan matrices. 
\begin{theorem}[\cite{LWY}]\label{A case} For the system \er{toda2} of type $A_n$, the $U_1$ is expressed by 
\begin{equation}\label{form}
\begin{split}
e^{-U_1} = |z|^{-2\alpha_1} \Big( \lambda_0 + \sum_{i=1}^n \lambda_i |P_i(z)|^2 \Big), \quad \text{where}\\
P_i(z) = z^{\mu_1 + \cdots + \mu_i} + \sum_{j=0}^{i-1} c_{ij} z^{\mu_1 + \cdots + \mu_j},
\end{split}
\end{equation}
$\mu_i = 1 + \gamma_i >0$, $c_{ij}$ are complex numbers, and the $\lambda_i > 0$ satisfy
\begin{equation}\label{prod}
\lambda_0 \cdots \lambda_n = \prod_{1\leq i\leq j\leq n} \Big( \sum_{k=i}^j \mu_k \Big)^{-2}.
\end{equation}
Furthermore for $j<i$,
\begin{equation}\label{c=0}
c_{ij}=0 \quad \text{if }\mu_{j+1} + \cdots + \mu_i \notin \bn.
\end{equation}
\end{theorem}

The major tool in \cite{LWY} is the following characterization of the components of $e^{-U_1}$. This is proved in Steps 1 and 2 of the proof of the main theorem 5.1 in \cite{LWY}. 
\begin{proposition}[\cite{LWY}]\label{char exp}
For the system \er{toda2} of type $A_n$, the $U_1$ is expressed by 
\begin{equation}
\label{use M}
e^{-U_1} = \sum_{i,j=0}^n m_{ij} \bar{f_i} {f_j},
\end{equation}
where $M=(m_{ij})_{i, j=0}^n$ is a Hermitian matrix, and $f_i(z)= z^{\beta_i}$ with 
\begin{equation}
\label{exponents}
\beta_0 = -\alpha_1, \quad \beta_i = \alpha_i - \alpha_{i+1} + i \ \ \text{for }1\leq i\leq n-1,\quad \beta_n = \alpha_n + n.
\end{equation}
\end{proposition}



\begin{proof}[New proof of Prop \ref{char exp} using \cite{N2}] The strategy for the proof is as follows. For the reader's convenience, we first repeat some estimates from \cite{LWY} using \cite{BM}. Then we present the formulas in \cite{N2} for computing the $W$-symmetries, also called characteristic integrals $W_j$. Then we use the above estimates to quickly show that the $W_j$ have simple forms, and this further implies that $e^{-U_1}$ has simple components. 

Following \cite[Eq. (5.10)]{LWY}, introduce 
\begin{equation}\label{def V}
V_i = U_i - 2\alpha_i \log |z|, \quad 1\leq i\leq n.
\end{equation}
Then system \er{toda2} becomes
$$
\begin{cases}
\displaystyle\Delta V_i = -4 |z|^{2\gamma_i} \exp\Big(\sum_{j=1}^n a_{ij} V_j\Big)\\
\displaystyle \int_{\br^2} |z|^{2\gamma_i} \exp\Big(\sum_{j=1}^n a_{ij} V_j\Big) \, dx <\infty.
\end{cases}
$$
As $\gamma_i > -1$, applying Brezis-Merle's argument in \cite{BM}, we have $V_i \in C^{0, \alpha}$ on $\bc$ for some $\alpha \in (0, 1)$ and they are upper bounded over $\bc$. Furthermore 
\begin{equation}
\label{order of V}
\begin{split}
\p_z^k V_i(z) &= O(1+|z|^{2+2\gamma_i -k}) \quad\text{ near } 0,\\
\p_z^k V_i(z) &= O(|z|^{-k}) \quad\text{ near } \infty, \forall\, k\geq 1.
\end{split}
\end{equation}
Now \er{def V} clearly implies that 
\begin{equation}
\label{order of U}
\begin{split}
\p_z^k U_i(z) &= O(|z|^{-k}) \quad\text{ near } 0, \forall\, k\geq 1,\\
\p_z^k U_i(z) &= O(|z|^{-k}) \quad\text{ near } \infty, \forall\, k\geq 1.
\end{split}
\end{equation}

For the Toda system \er{toda2} of type $A$, the $W$-symmetries are computed in \cite[Example 3.1]{N2} by the expansion
\begin{equation}
\label{my way}
\begin{split}
\cl &= (\p_z - U_{n, z}) (\p_z + U_{n, z} - U_{n-1, z}) \cdots (\p_z + U_{2, z} - U_{1, z}) (\p_z + U_{1, z}) \\
&=  \p_z^{n+1} + \sum_{j=1}^n W_j \p_z^{n-j}.
\end{split}
\end{equation}
(Here we have used the symmetry of the $A_n$ system with respect to $U_i$ and $U_{n+1-i}$ in the original formula (3.1) in \cite{N2} to conform to \cite{LWY}. For $1\leq j\leq n$, our $W_j = Z_{n+1-j}$ in \cite[Lemma 5.2]{LWY}.) The $W_j$ are differential polynomials in the $U_i$, that is, they are polynomials in the $U_i$ and their derivatives with repsect to $z$. The $W_j$ are called characteristic integrals and denoted by $I_j$ in \cite{N2}, since for solutions $U_i$ to the Toda system \er{toda2}, they satisfy the property 
\begin{equation}\label{char prop}
W_{j, \bar z}=0, \quad 1\leq j\leq n.
\end{equation} 
The $W_j$ are also called $W$-symmetries in \cite{LWY}. 
For a differential mononomial in the $U_i$, we call by its \emph{degree} the sum of the orders of differentiation multiplied by the algebraic degrees of the corresponding factors. For example $U_{1, zz} - U_{1,z}^2$ has a homogeneous degree $2$. 
Clearly from the formula \er{my way}, we see that $W_j$ has homogeneous degree $j+1$ for $1\leq j\leq n$. 

Therefore \er{order of U} and \er{char prop} imply that $z^{j+1} W_j$ is holomorphic and bounded on $\bc^*$, and hence is a constant by Liouville's theorem. That is, for $1\leq j\leq n$, 
\begin{equation}\label{Ws}
W_j = \frac{w_j}{z^{j+1}} \quad \text{ on }\bc^*
\end{equation} 
for some number $w_j$. 

Since $\gamma_i>-1$, Eqs. \er{order of V} and \er{order of U} imply that near 0, the $W_j$ can be computed using the 
$$2\alpha_i \log |z| = \alpha_i (\log z + \log \bar z)$$
summand of $U_i$ in \er{def V}. That is, we have 
\begin{equation}
\label{alpha}
\begin{split}
&\quad \Big(\p_z - \frac{\alpha_n}{z}\Big) \Big(\p_z + \frac{\alpha_{n} - \alpha_{n-1}}{z}\Big) \cdots \Big(\p_z + \frac{\alpha_2 - \alpha_1}{z}\Big) \Big(\p_z + \frac{\alpha_1}{z}\Big) \\
&=  \p_z^{n+1} + \sum_{j=1}^n W_j \p_z^{n-j}.
\end{split}
\end{equation}

This totally determines the $W_j = \frac{w_j}{z^{j+1}}$, and in particular the $w_j$ as real numbers. Now from \er{my way}, it is clear that $\cl e^{-U_1} = 0$. Call $f=e^{-U_1}$, and we have 
$$
f^{(n+1)} + \sum_{j=1}^n \frac{w_j}{z^{j+1}} f^{(n-j)} = 0.
$$
The fundamental solutions of this Cauchy-Euler equation are of the form $z^{\beta_i}$ for $0\leq i \leq n$. We use \er{alpha} to quickly see that the $\beta_i$ are those specified in \er{exponents}. 

Since we require that our solutions are real, we see that $e^{-U_1}$ must have the form in \er{use M} using a Hermitian matrix. 
\end{proof}

\begin{remark} \cite{LWY} obtained the characteristic exponents in \er{exponents} by constructing the solutions to the Toda system \er{toda2} and hence their method is indirect. Here the structure theory of the $W$-symmetries developed by this author makes this proof more transparent. Actually, in Remark 3.4 of \cite{LWY}, the authors mentioned that the explicit formulas of these invariant are hard to find in the literature, and they constructed these invariants by a different method in \cite{LWY}. 
\end{remark}

In the remainder of this section, we cast the results in \cite{LWY} in the framework of iterated integral solutions of \cite{L}. There are no new results or proofs here, except that the viewpoint is to this author more conceptual. 
Using $\mu_i = \gamma_i + 1$ for $1\leq i\leq n$, it is easy to see that \er{exponents} is equivalent to 
\begin{equation}\label{new exp}
\beta_0 = -\alpha_1,\quad \beta_i - \beta_0 = \mu_1 + \cdots +\mu_i,\ \forall\, 1\leq i\leq n.
\end{equation}

In the literature of Toda systems as integrable systems, especially in \cite{L}, there is a general construction of solutions to the Toda systems of type $A$ given $n$ arbitrary functions $\phi_1(z), \dots, \phi_n(z)$ of $z$ and $n$ arbitrary functions of $\bar z$. First we formally and locally define the intreated integrals
\begin{multline}\label{iteint}
\sigma_0(z) := 1,\ \sigma_1(z) := \int^z_0 \phi_1(z_1)\, dz_1, \\
\sigma_i(z) :=\int^z_0 \phi_{1}(z_1)\,dz_1\cdots\int^{z_{i-2}}_0\phi_{i-1}(z_{i-1})\,dz_{i-1}\int^{z_{i-1}}_0\phi_{i}(z_i)\,dz_i, \text{ for a general }i.
\end{multline}
Also define 
\begin{equation}\label{bottom}
\xi(z) := \prod_{j=1}^n \phi_j(z)^{a^{1j}} = \prod_{j=1}^n \phi_j(z)^{\frac{n+1-j}{n+1}},
\end{equation}
where the $a^{1j}= \frac{(n+1-j)}{n+1}$ are from the first row of the inverse Cartan matrix to \er{Cartan-A}. 
Define 
$$
\nu_i(z) := \frac{\sigma_i(z)}{\xi(z)},\quad 0\leq i\leq n.
$$
The important property \cite{L, LS-book, N2} is that the Wronskian
\begin{equation}\label{ww1}
W(\nu_0, \nu_1, \dots, \nu_n) = 1
\end{equation}
for $n$ arbitrary functions $\phi_1, \cdots, \phi_n$.  

Since we are only interested in real solutions, the functions of $\bar z$ are hence taken to be just the conjugates. Then \cite{L} asserts that  
\begin{equation}
\label{Leznov}
e^{-U_1} = |\nu_0|^2 + |\nu_1|^2 + \cdots + |\nu_n|^2 = \frac{1 + |\sigma_1(z)|^2 + \cdots + |\sigma_n(z)|^2}{ \big| \xi(z) \big |^2}
\end{equation}
defines a solution to the Toda systems of type $A_n$. 

Very neatly put, Proposition \ref{char exp} through the $W$-symmetries just determines that the integrand functions in the iterated integral scheme of Leznov and Saveliev \cites{L, LS, LS-book} are just the following functions
\begin{equation}
\label{integrands}
\phi_i(z) = z^{\gamma_i}, \quad \forall\, 1\leq i\leq n.
\end{equation}
Then \er{iteint} and \er{bottom} become (at least after the branch cut on $\bc\backslash \{x \,|\, x_1\leq 0\}$) 
\begin{equation}\label{new sigma}
\sigma_i(z) = \frac{z^{\mu_1 + \cdots +\mu_i}}{\mu_i(\mu_i + \mu_{i-1}) \cdots (\mu_i + \cdots + \mu_1)}, \quad \forall\, 0\leq i\leq n,
\end{equation}
and 
\begin{equation}\label{find xi}
\xi(z) = \prod_{j=1}^n z^{a^{1j} \gamma_j} = z^{\alpha_1}.
\end{equation}
Using \er{new exp}, we see that 
\begin{equation}\label{scale}
\nu_i(z) = \frac{\sigma_i(z)}{\xi(z)} = \frac{1}{\mu_i\cdots(\mu_i+\cdots+\mu_1)} z^{\beta_i} =\chi_i f_i,\quad 0\leq i\leq n.
\end{equation}
That is, $\nu_i$ is the same as the $f_i$ in Proposition \ref{char exp} up to a scale 
$\displaystyle\chi_i = \prod_{j=1}^i \Big(\sum_{k=j}^i \mu_k\Big)^{-1}$.

Therefore \er{Leznov} becomes 
\begin{equation}\label{radial}
e^{-U_1} = {|z|^{-2\alpha_1}}\Big(\chi_0^2 + \sum_{i=1}^n \chi_i^2 |z^{\mu_1 + \cdots + \mu_i} | ^{2} \Big).
\end{equation}
We note that corresponding to \er{prod}, we have 
$$
\chi_0^2 \cdots \chi_n^2 = \prod_{1\leq i\leq j\leq n} \Big( \sum_{k=i}^j \mu_k \Big)^{-2}.
$$

The solution \er{radial} corresponds to a radial solution of the Toda system \er{toda2}, and there are other solutions. The extra freedom comes from the Hermitian matrix $M$ in terms of the $f_i$ from \er{use M}. From our point of view, the version using the $\nu_i$ in \er{scale} is more natural and has better properties later on. Therefore we define 
\begin{equation}\label{use H}
e^{-U_1} = \sum_{i,j=0}^n h_{ij} \bar{\nu_i} \nu_j,
\end{equation}
where
$
H = (h_{ij})_{i, j=0}^n = X M X,
$
and $X := \on{diag}(\chi_0, \chi_1, \cdots, \chi_n)^{-1}$ is the diagonal matrix. 

Then $H$ is Hermitian and must have determinant 1, i.e. $H\in SL(n+1, \bc)$, for \er{use H} to define a solution of the Toda system \er{toda2} of type $A_n$.  See \cite{L, N1}, and this corresponds to 
condition \er{prod} from \cite{LWY}. 

Furthermore as shown in \cite{LWY}, $M$ and hence $H$ must be positive definite. 
The Cholesky decomposition \cite{GvL} then states that
\begin{equation}\label{UL}
H = B^\dag B,
\end{equation}
where $B$ is a lower-triangular matrix with diagonal entries $b_{ii}>0$ and $B^\dag$ denotes the conjugate transpose $\bar B^t$. Clearly $B = \Lambda C$ where $\Lambda$ is the diagonal of $B$ and $C$ is lower-triangular with 1's on its diagonal. 

This then gives the form \er{form} of the solutions in Theorem \ref{A case}. For the solution to be well-defined on $\bc^*$, it is easy to see the condition \er{c=0} on the $c_{ij}$. 



We summarize the results for the $A_n$ case in the following form, which provides more details to Theorem \ref{A in intro}. 
\begin{proposition} All the solutions to the Toda system \er{toda2} of type $A_n$ are given, in terms of the first unknown $U_1$, in the following way. Using the family of functions $\phi_i(z)$ as in \er{integrands}, define the iterated integrals $\sigma_i(z)$ as in \er{new sigma} and the $\nu_i(z) = \frac{\sigma_i(z)}{\xi(z)}$ as in \er{scale} with $\xi(z)=z^{\alpha_1}$ as in \er{find xi}. Consider the vector of functions 
\begin{equation}\label{vect}
{\bnu} = (\nu_0, \nu_1, \cdots, \nu_n)^t.
\end{equation}

Let $G=SL_{n+1}(\bc)$ and let $G=KAN$ be its Iwasawa decomposition with $K=SU(n+1)$, $A$ the abelian group of diagonal matrices with positive diagonal entries and determinant 1, and $N$ the unipotent group of unipotent lower-triangular matrices. Corresponding to \er{c=0}, let $N_\Gamma$ be the subgroup of $N$ corresponding to the following system of roots $\Delta_\Gamma \subset \Delta^+$
\begin{equation}\label{A cond}
\Delta_\Gamma := \Big\{\sum_{k=j+1}^i \tau_k \,\Big|\, \sum_{k=j+1}^i \gamma_k \in \bz,\ j<i \Big\}. 
\end{equation}
Here the $\tau_k$ are the system of simple roots for $A_n$, and this system $\Delta_\Gamma$ is closed under addition. (See Definition \ref{in roots} for more.)

Then the solution space of \er{toda2} is parametrized by $AN_\Gamma$. In particular for $\Lambda\in A$ and $C\in N_\Gamma$, 
$$
U_1 = -\log | \Lambda C \bnu |^2
$$
defines a solution. 
\end{proposition}

\section{Minors of symplectic and orthogonal matrices}\label{minors}

In preparation for our classification of solutions to the $C_n$ and $B_n$ Toda systems, we need some algebraic properties of symplectic and odd-dimensional orthogonal matrices. More precisely, we will
show that such matrices in the special linear group are characterized by some equalities among their minors.

First we define the following rank-$k$ matrix 
\begin{equation}\label{def j}
J_{k}=\begin{pmatrix}
& & & & 1\\
& & & -1 & \\
& & 1 & & \\
& \iddots & & & \\
(-1)^{k-1} & & & &\\
\end{pmatrix},
\end{equation}
with $\pm 1$ along the secondary diagonal. 
When $k=2n$ is even, $J_{2n}$ is skew-symmetric and defines a non-degenerate skew-symmetric bilinear form on $\bc^{2n}$. 
When $k=2n+1$ is odd, $J_{2n+1}$ is symmetric and defines a non-degenerate symmetric bilinear form on $\bc^{2n+1}$. 
Note that 
\begin{equation}\label{J-1}
J^{-1} = (-1)^{k-1} J.
\end{equation}

We define the complex symplectic group and complex special orthogonal group in dimensions $2n$ and $2n+1$ with respect to the $J_{2n}$ or $J_{2n+1}$, that is, 
\begin{align}
Sp(2n, \bc) &:= \{A\in GL(2n, \bc)\,|\, A^t J_{2n} A = J_{2n}\} \label{symp}\\
SO(2n+1, \bc) &:= \{A\in SL(2n+1, \bc)\,|\, A^t J_{2n+1} A = J_{2n+1}\} \label{orth}.
\end{align}
A matrix $A$ in $Sp(2n, \bc)$ is called a symplectic matrix, and it is well-known that $\det A = 1$. A matrix $A$ in $SO(2n+1, \bc)$ is required to have $\det A = 1$ (in general the determinant is $\pm 1$), and is called a special odd-dimensional orthogonal matrix, or just an orthogonal matrix for short. 

\begin{remark} There are other more conventional choices of the matrices defining the bilinear forms, and our versions can be related to them easily. For example in the symplectic case, to transform from $J_{2n}$ in \er{def j} to the more conventional 
$$
\Omega_{2n}=\begin{pmatrix}
& & & & & 1\\
& & & & \iddots & \\
& & & 1 & & \\
& & -1 & & & \\
& \iddots & & & & \\
-1 & & & & &\\
\end{pmatrix},
$$
we define a diagonal matrix $Q$ with 
\begin{equation}\label{def Q}
Q_{ii} = \begin{cases} (-1)^{i-1} & 1\leq i\leq n \\
1 & n+1\leq i\leq 2n
\end{cases},
\end{equation}
and then 
$$
Q J_{2n} Q = \Omega_{2n}.
$$
Therefore if $A$ in \er{symp} is symplectic with respect to $J_{2n}$, then $QAQ$ is symplectic with respect to $\Omega_{2n}$.  

Similarly in our orthogonal case, the $Q$ in \er{def Q} with the $2n$ replaced by $2n+1$ transforms our $J_{2n+1}$ in \er{def j} to the more conventional secondary diagonal matrix $\Theta$ with $\Theta_{i, 2n+2-i}=1$ for $1\leq i\leq 2n+1$. 
\end{remark}

This author discovered the bilinear form in \er{def j} in the work \cite{N1}, and they are the most relevant for the study of Toda systems as they are naturally related to the differential properties of iterated integrals (see Proposition \ref{mine} below). Furthermore the related algebraic properties of minors are also nicer since no signs come up at all. Proposition \ref{prop-minors} below is an improvement of \cite[Prop. 3.4]{N1}, which only treated the case of principal minors. The corresponding result using the more conventional bilinear forms for general minor would involve many signs that are hard to pin down (see \cite[Remark 3.5]{N1}). 

Let us first introduce some notation. Let $\{e_i\}_{i=1}^k$ be the standard basis of $\bc^k$. 
Let $A\in SL(k,\bc)$ be a matrix with determinant 1. We are concerned with the cases where $k=2n$ and $A$ is symplectic as in \er{symp}, or $k=2n+1$ and $A$ is special orthogonal as in \er{orth}. 
The matrix $A = (a_{ij})$ defines a linear transformation $\ca$ on $\bc^k$ by 
$$
\ca(e_j) = \sum_{i=1}^k a_{ij} e_i,\quad 1\leq j\leq k.
$$
Naturally $\ca$ further induces linear transformations on the exterior powers $\wedge^* \bc^k$, for which we still use the notation. 

Let $S\subset \{1,2,\cdots,k\}$ be a subset of indices. Denote the number of elements in $S$ by $m$ and write 
\begin{equation}\label{write S}
S=\{s_1,s_2,\cdots,s_m\} \quad \text{with }s_1<s_2<\cdots<s_m.
\end{equation}
Let $T$ be another subset with the same cardinality as $S$, and write 
\begin{equation}\label{write T}
T=\{t_1,t_2,\cdots,t_m\} \quad \text{with }t_1<t_2<\cdots<t_m.
\end{equation}
Let $A_{S,T}$ denote the minor of $A$ with row indices in $S$ and column indices in $T$, that is, 
$$
A_{S,T}=\det(a_{s_i t_j})_{i,j=1}^m. 
$$

The efficient way of viewing minors is through exterior products. Insisting on ordering the elements increasingly as in \er{write S}, we define 
$$
e_S:=e_{s_1}\wedge e_{s_2}\wedge \cdots \wedge e_{s_m}\in \wedge^m \bc^k.
$$
Then the minor $A_{S,T}$ is nothing but the entry of the matrix for the induced transformation on $\wedge^m \bc^n$ in the $(e_S,e_T)$ position, that is,
\begin{equation}\label{coeff}
\ca(e_T) = A_{S,T} e_S + \text{other terms}.
\end{equation}
Note that for the top exterior form, $A\in SL(k, \bc)$ implies that 
\begin{equation}\label{top}
\ca(e_1\wedge\cdots\wedge e_k)=(\det A)(e_1\wedge \cdots\wedge e_k)=e_1\wedge \cdots\wedge e_k.
\end{equation}

Denote the complement of $S$ in $\{1,2,\cdots,k\}$ by $\bar S$.  
Let $\iota$ denote the inversion of the indices in $\{1,2,\cdots,k\}$, that is, 
\begin{equation}\label{inversion}
\iota(j)=k+1-j, \quad 1\leq j\leq k.
\end{equation}
The notation $\iota(S)$ denotes the set $\{\iota(s)\,|\, s\in S\}$, where we need to re-order the elements to make them increasing. 

\begin{proposition}\label{prop-minors}
For $A\in Sp(2n, \bc)$ or $A\in SO(2n+1, \bc)$ and any two subsets $S$ and $T$ of $\{1, 2, \cdots, 2n(+1)\}$ of the same cardinality, we have 
\begin{equation}\label{neat}
A_{ST} = A_{\iota(\bar S) \iota(\bar T)}.
\end{equation}

On the other hand, if $A\in SL(k, \bc)$ and \er{neat} is satisfied for all $S$ and $T$ with cardinality 1, that is, for all $1\leq s, t\leq k$ 
\begin{equation}\label{minor}
A_{st} = A_{\bar{\iota(s)},\bar{\iota(t)}},
\end{equation}
where $A_{st}=a_{st}$ is the element of $A$ in the $(s, t)$ position, and $A_{\bar{\iota(s)},\bar{\iota(t)}}$ is the minor of $A$ after deleting the $\iota(s)$-th row and $\iota(t)$-th column, 
then the matrix $A$ is either symplectic or orthogonal depending on the parity of $k$. 
\end{proposition}

\begin{proof} The two cases are treated in the same way with $k$ denoting either $2n$ in the symplectic case or $2n+1$ in the orthogonal case. We write $J$ for $J_k$, and denote its corresponding linear transformation by $\cj$. Then from \er{symp} and \er{orth}, we have that $A$ is symplectic or orthogonal if and only if 
$A^{-1} = J^{-1} A^t J,$
or in terms of the linear transformations on $\bc^k$, 
\begin{equation}\label{inverse}
\ca^{-1} = \cj^{-1} \ca^t \cj,
\end{equation}
where $\ca^t$ denote the linear transformation defined by $A^t$. 

We will prove \er{neat} in the following form
\begin{equation}\label{rewrite}
A_{\bar S \bar T} = A_{\iota(S) \iota(T)}.
\end{equation}
We denote the common cardinality of $S$ and $T$ by $m$ and order their elements according to \er{write S} and \er{write T}. 

As preparations, we compute the following two signs
\begin{align}
e_{\bar S} \wedge e_{S} &=\varepsilon_1\, e_1 \wedge \cdots \wedge e_k, \label{SS}\\
\cj e_S &= \varepsilon_2\, e_{\iota(S)}. \label{JS}
\end{align}

Such calculations are elementary. First we have $e_{\bar S} \wedge e_S = (-1)^{m(k-m)} e_S \wedge e_{\bar S}$. To calculate the sign for $e_S\wedge e_{\bar S}$, we need to move the $e_1, \cdots, e_{s_1-1}$ over the $m$ elements of $e_{s_i}$, and so on. Therefore
\begin{align*}
\ve_1 &= (-1)^{m(k-m) + (s_1 - 1) m + (s_2 -1 - s_1)(m-1) + \cdots + (s_m -1 - s_{m-1})} \\
&= (-1)^{m(k-m) + s_1 + \cdots + s_m - (m + \cdots + 1)} = (-1)^{m(k-m) + \sum_{i=1}^m s_i - \frac{m(m+1)}{2}}.
\end{align*}
Since $\cj e_j = (-1)^{k-j} e_{\iota(j)}$, and we need a re-ordering of $m$ elements in $\iota(S)$, we have 
\begin{align*}
\ve_2 &= (-1)^{ k-s_1 + \cdots + k-s_m + \frac{m(m-1)}{2} }\\
&= (-1)^{ mk - \sum_{i=1}^m s_i + \frac{m(m-1)}{2} }.
\end{align*}
Note that 
\begin{equation}\label{=1}
\ve_1 \ve_2 = 1.
\end{equation}

Now we compute as follows.
\begin{equation}\label{first}
\begin{split}
A_{\bar S \bar T}\, e_1 \wedge \cdots \wedge e_k &= \ve_1 A_{\bar S \bar T}\, e_{\bar S} \wedge e_{S} = \ve_1 (\ca e_{\bar T})\wedge e_S\\
& = \ve_1 \ca( e_{\bar T} \wedge (\ca^{-1} e_S)) = \ve_1 e_{\bar T}\wedge (\ca^{-1} e_S),
\end{split}
\end{equation}
where we have used \er{coeff} and 
\er{top}. 

Using \er{inverse}, we see that the above can be continued as 
\begin{equation}\label{second}
\begin{split}
\ve_1 e_{\bar T} \wedge (\cj^{-1} \ca^t \cj e_S) &= \ve_1 \ve_2 e_{\bar T} \wedge (\cj^{-1} \ca^t e_{\iota(S)})\\
&= A_{\iota(S)\iota(T)} e_{\bar T} \wedge \cj^{-1} e_{\iota(T)}\\
&= (-1)^{(k-1)(k-m)} A_{\iota(S)\iota(T)} e_{\bar T} \wedge \cj e_{\iota(T)},
\end{split}
\end{equation}
by \er{=1}, \er{J-1} and $e_{\iota(T)} \in \wedge^{k-m} \bc^k$. 

Adapting \er{JS} to $\iota(T)$ and \er{SS} to $T$, 
we have 
\begin{align*}
e_{\bar T} \wedge \cj e_{\iota(T)} &= (-1)^{mk - \sum_{i=1}^m (k+1-t_i) + \frac{m(m-1)}{2}} e_{\bar T} \wedge e_{T}\\
&= (-1)^{\sum_{i=1}^m {t_i} + \frac{m(m+1)}{2} }e_{\bar T} \wedge e_T\\
&= (-1)^{m(k-m)} e_1\wedge \cdots \wedge e_k.
\end{align*}

Therefore \er{second} can be continued as  
\begin{equation}\label{last}
\begin{split}
(-1)^{(k-1)(k-m)} A_{\iota(S)\iota(T)} e_{\bar T} \wedge \cj e_{\iota(T)} &= (-1)^{(k+m-1)(k-m)} A_{\iota(S)\iota(T)} e_1\wedge \cdots \wedge e_k\\
&= A_{\iota(S)\iota(T)} e_1\wedge \cdots \wedge e_k.
\end{split}
\end{equation}
The combination of \er{first}, \er{second} and \er{last} gives \er{rewrite}. 

For the converse direction, note that \er{minor} is just \er{rewrite} with $S=\{\iota(s)\}$ and $T=\{\iota(t)\}$. In this case, 
combining Eqs. \er{second}, \er{last},  \er{minor}, and \er{first}, we get 
$$
\ve_1 e_{\bar T} \wedge (\cj^{-1} \ca^t \cj e_S) = \ve_1 e_{\bar T}\wedge (\ca^{-1} e_S)
$$
for all $S=\{\iota(s)\}$ and $T=\{\iota(t)\}$. Therefore 
$$
A^{-1} = J^{-1} A^{t} J.
$$
By definitions \er{symp} and \er{orth}, we see that $A$ is symplectic or orthogonal if it satisfies \er{minor}. 
\end{proof}

\section{Reduction of Toda systems of types $C_n$ and $B_n$ to $A_{k-1}$}\label{red}

It is well-known \cite{L, ALW} that the $C_n$ and $B_n$ Toda systems can be treated as the $A_{2n-1}$ and $A_{2n}$ Toda systems with symmetries. More precisely, we will show in this section that the solutions to Toda systems \er{toda2} of types $C_n$ and $B_n$ correspond to the solutions of Toda systems \er{toda3} of types $A_{2n-1}$ and $A_{2n}$ with the symmetry condition \er{sym for gamma} and with the symmetry requirement \er{sym for U}. 
The symmetry corresponds to the outer automorphism of the Lie algebra $A_{k-1}$, which graphically is represented by the symmetry of its Dynkin diagram. 

In this section, we give no details for the verification of the assertions since they are elementary. The $B_n$ case is slightly more complicated than the $C_n$ case. Note that for the $B_n$ case, $\delta_{i, n}$ denotes the Kronecker delta. 

\begin{lemma}[$C_n$ reduction] The $u_i$ for $1\leq i \leq n$ satisfy \er{toda} for the $C_{n}$ Toda system with parameters $\gamma_i$ for $1\leq i\leq n$ iff the $\t u_i$ for $1\leq i\leq 2n-1$ defined by 
\begin{equation*}\label{reduction}
\t u_i = \t u_{2n-i} = u_i,\quad 1\leq i\leq n,
\end{equation*}
satisfy \er{toda} for the $A_{2n-1}$ Toda system with parameters $\t\gamma_i$ for $1\leq i\leq 2n-1$ defined by 
\begin{equation*}\label{C_n gamma}
\t \gamma_i = \t \gamma_{2n-i} = \gamma_i,\quad 1\leq i\leq n. 
\end{equation*}
\end{lemma}

\begin{lemma}[$B_n$ reduction] The $u_i$ for $1\leq i \leq n$ satisfy \er{toda} for the $B_{n}$ Toda system with parameters $\gamma_i$ for $1\leq i\leq n$ iff the $\t u_i$ for $1\leq i\leq 2n$ defined by 
\begin{equation*}\label{reduction1}
\t u_i = \t u_{2n+1-i} = u_i + \delta_{i, n} \ln 2,\quad 1\leq i\leq n,
\end{equation*}
satisfy \er{toda} for the $A_{2n}$ Toda system with parameters $\t\gamma_i$ for $1\leq i\leq 2n$ defined by 
\begin{equation*}\label{B_n gamma}
\t \gamma_i = \t \gamma_{2n+1-i} = \gamma_i,\quad 1\leq i\leq n. 
\end{equation*}
\end{lemma}

Similar reductions can be done in the other version of the Toda systems \er{toda2}, and we can treat the two types in the same way as reductions of $A_{k-1}$ Toda systems. We have $k=2n$ in the $C_n$ case, and $k=2n+1$ in the $B_n$ case. 
Now consider the $A_{k-1}$ Toda system \er{toda2} as
\begin{equation}
\label{toda3}
\begin{cases}
\displaystyle
\Delta \t U_{i} + 4\exp\Big(\sum_{j=1}^{k-1}  a_{ij} \t U_j\Big) = 4\pi \t \alpha_i \delta_0,&1\leq i\leq k-1 \\
\displaystyle
\int_{\br^2} e^{\sum_j a_{ij}\t U_j} < \infty, & 1\leq i\leq k-1, 
\end{cases}
\end{equation}
together with the conditions
\begin{equation}\label{sym for gamma}
\t \gamma_i = \t \gamma_{k-i} = \gamma_i>-1,\quad 1\leq i\leq \hk,
\end{equation}
and the requirement 
\begin{equation}\label{sym for U}
\t U_i = \t U_{k-i},\quad 1\leq i\leq \hk.
\end{equation}

In the $C_n$ case,  under the condition \er{sym for gamma}, the $\t\alpha_i$ as in \er{def alpha} defined by the $\t\gamma_j$ using the inverse Cartan matrix of $A_{2n-1}$ for $1\leq i, j\leq 2n-1$ are related to the 
$\alpha_i$ defined by the $\gamma_j$ using the inverse Cartan matrix of $C_{n}$ for $1\leq i, j\leq n$ by 
\begin{equation}\label{C_n alp}
\t \a_i = \t \a_{2n-i} = \a_i,\quad 1\leq i\leq n. 
\end{equation}
Then the 
$$
U_i := \t U_i, \quad 1\leq i\leq n
$$
satisfy the $C_n$ Toda system \er{toda2} with parameters $\a_i$ for $1\leq i\leq n$ from \er{C_n alp}, and all the solutions to the $C_n$ system are obtained in this way. 

Now the $B_n$ case. Under the condition \er{sym for gamma}, the $\t\a_i$ as in \er{def alpha} defined by the $\t\gamma_j$ using the inverse Cantan matrix of $A_{2n}$ for $1\leq i, j\leq 2n$  are related to the $\alpha_i$ defined by the $\gamma_j$ using the inverse Cartan matrix of $B_{n}$ for $1\leq i, j\leq n$ by 
\begin{equation}\label{B_n alp}
\t\a_i = \t\a_{2n+1-i} = (1+\delta_{i, n})\a_i, \quad 1\leq i\leq n.
\end{equation}
Then the 
\begin{equation}\label{BU}
U_i := \frac{1}{1+\delta_{i, n}}(\t U_i - i\cdot \ln 2), \quad 1\leq i\leq n
\end{equation}
satisfy the $B_n$ Toda system \er{toda2} with parameters $\a_i$
for $1\leq i\leq n$ from \er{B_n alp}, and all solutions to the $B_n$ system are obtained in this way. Note that the negative coefficient vector of $\ln 2$, 
$$
\Big(1, 2, \cdots, n-1, \frac{n}{2}\Big)^t,
$$
is the last column vector of the inverse Cartan matrix of $B_n$ (cf. \er{Cartan-B}). 


\section{Imposing symmetry for solutions}\label{symmetry}

The classification result of \cite{LWY} applies to the system \er{toda3} with the symmetry condition \er{sym for gamma} for the $\t\gamma_i$. 
As expected, the symmetry requirement \er{sym for U} for the solutions enforces some group structure, and we carry this out in this section. 

We follow \cite{LWY} and Section \ref{revisit} to present the solutions of \er{toda3} with \er{sym for gamma}. The integrand functions in \er{integrands} become 
\begin{equation}\label{sym for phi}
\phi_i(z) = \phi_{k-i}(z) = z^{\gamma_i}, \quad \forall\, 1\leq i\leq \Big\lfloor\frac{k}{2}\Big\rfloor.
\end{equation}
Then \er{find xi} becomes $\xi(z)=z^{\t \a_1} = z^{\a_1}$ by \er{C_n alp} and \er{B_n alp}. 
Define $\sigma_i(z)$ by \er{new sigma} and $\nu_i$ by \er{scale} for $0\leq i\leq k-1$. The vector of functions \er{vect} becomes
\begin{equation}\label{new nv}
{\bnu} = (\nu_0, \nu_1, \cdots, \nu_{k-1})^t.
\end{equation}

\begin{remark}\label{repeation} More concretely, in the $C_n$ case the integrand functions and the $\xi(z)$ are 
\begin{gather*}
(\phi_1, \cdots, \phi_{2n-1}) = (z^{\gamma_1}, \cdots, z^{\gamma_n}, \cdots, z^{\gamma_1}),\\
\xi(z) = z^{\gamma_1 + \cdots + \gamma_{n-1} + \frac{1}{2}\gamma_n}.
\end{gather*}
In the $B_n$ case they are
\begin{gather*}
(\phi_1, \cdots, \phi_{2n}) = (z^{\gamma_1}, \cdots, z^{\gamma_n}, z^{\gamma_n}, \cdots, z^{\gamma_1}),\\
\xi(z) = z^{\gamma_1 + \cdots + \gamma_{n-1} + \gamma_n}.
\end{gather*}
\end{remark}

Then the solution to \er{toda3} is defined by \er{use H} which becomes 
\begin{equation}\label{new H}
e^{-\t U_1} = \sum_{i,j=0}^{k-1} h_{ij} \bar{\nu_i} \nu_j,
\end{equation}
where $H=(h_{ij})_{i,j=0}^{k-1}$ is Hermitian and has determinant 1. 

\begin{proposition}\label{main} Formula \er{new H} defines a solution of \er{toda3} with \er{sym for gamma} that satisfies the requirement \er{sym for U} if and only if 
the Hermitian matrix $H$ is symplectic or orthogonal as in \er{symp} or \er{orth}. 
\end{proposition}

We first do some preparations for the proof. Let $W=W(\bnu)$ denote the Wronskian matrix of $\bnu$ with respect to $z$, that is, 
\begin{equation}\label{Wrons}
W = \begin{pmatrix}
\bnu &
\bnu' &
\cdots &
\bnu^{(k-1)}
\end{pmatrix}.
\end{equation}
We know that $W\in SL(k, \bc)$ by \er{ww1}. Now consider the big matrix 
\begin{equation}\label{big S}
R :=  {W}^\dag H {W}
\end{equation}
Then \er{new H} says that 
\begin{equation}\label{U1}
e^{-\t U_1} = R_{1,1}.
\end{equation}
The general theory \cite{L, N1, LWY}  of Toda system \er{toda3} says that in general for $1\leq m\leq k-1$, we have 
\begin{equation}\label{Um}
e^{-\t U_m} = R_{\ul{m}, \ul{m}},
\end{equation}
the first principal minor of $R$ of rank $m$ where $\ul{m}$ denotes the set $\{1, \cdots, m\}$. 

Our way to prove Proposition \ref{main} is by applying Proposition \ref{prop-minors} in both directions. First we want to show that the Wronskian matrix in \er{Wrons} has special properties and, loosely speaking, is half symplectic or orthogonal. 

\begin{proposition}\cite[Prop. 5.13]{N1} \label{mine} Let $\bnu^{(i)}$ denote the $i$th derivative of $\bnu$ \er{new nv}. Then 
\begin{align}
(\bnu^{(i)})^t J \bnu^{(j)} &= 0, \qquad\text{if }i+j<k-1,\label{the zeros}\\
(\bnu^{(i)})^t J \bnu^{(j)} &= (-1)^{i}, \qquad\text{if }i+j=k-1,\label{the ones}
\end{align}
where $J=J_k$ is as in \er{def j}. 
\end{proposition}

\begin{remark} The above proposition is proved in \cite{N1} as some differential property of iterated integrals. In \cite[Prop. 5.13]{N1}, the $n$ is our $k-1$ here, the $F^t$ is our $\bnu$, and the swap function $s$ has no effect in our case by our symmetry condition \er{sym for phi}. 

Furthermore, by differentiating \er{the ones} and by a quick induction, we see that 
\begin{equation}\label{one more}
(\bnu^{(i)})^t J \bnu^{(j)} = 0, \qquad\text{if }i+j=k.
\end{equation}
See \cite[Eqs. (4.11),(4.14)]{N1}.
\end{remark}

Therefore for the Wronskian matrix $W$ \er{Wrons}, we have that 
\begin{equation}\label{pair}
P := W^t J W = \begin{pmatrix}
& & & & 1\\
& & & -1 & 0\\
& & 1 & 0& p_{2, k-1}\\
& \iddots & \iddots&\iddots & \vdots\\
(-1)^{k-1} & 0& p_{k-1, 2} & \cdots &p_{k-1, k-1} \\
\end{pmatrix},
\end{equation}
where $p_{i, j}=(\bnu^{(i)})^t J \bnu^{(j)}$. The $P$ is skew-symmetric or symmetric depending on $k$. 

A Gram-Schmidt procedure can be applied to the column vectors of the Wronskian matrix $W$ to make it symplectic or orthogonal. 

\begin{lemma}\label{GS} There exists a unipotent upper-triangular matrix $U$ such that 
$$
(WU)^t J (WU) = J.
$$  
\end{lemma}

\begin{proof} There are several ways of doing the normalization. One particularly nice way is by the following induction scheme. 

We first want to get a normalized basis for the plane spanned by the first and the last vectors, i.e. by $\bnu$ and $\bnu^{(k-1)}$. 
By \er{the zeros} and \er{the ones}, the only non-normalized quantity is $p_{k-1,k-1} = (\bnu^{(k-1)})^t J \bnu^{(k-1)}$. This number is nonzero only in the orthogonal case, i.e. when $k$ is odd. Therefore $v_0 = \bnu$ and $v_{k-1} = \bnu^{(k-1)} - \frac{p_{k-1,k-1}}{2} \bnu$ make a normalized basis of this plane. 

With respect to this plane, the other vectors may have nonzero pairing $$ (\bnu^{(i)})^t J v_{k-1} = (\bnu^{(i)})^t J \bnu^{(k-1)} =p_{i, k-1}$$ for $1\leq i\leq k-2$ (actually for $2\leq i\leq k-2$ by \er{one more}). Then we define 
$
v_i := \bnu^{(i)} - p_{i,k-1} \bnu
$
such that $(v_i)^t J v_{k-1} = 0$ for $1\leq i\leq k-2$. 

Then we will go to next plane spanned by $v_1$ and $v_{k-2}$ and continue this procedure to normalize the vectors $v_{k-2}$ and $v_2, \cdots, v_{k-3}$. 
Note the sign $(\bnu')^t J \bnu^{(k-2)} = -1$ from \er{the ones} will come into the formulas.

Therefore through induction, we have constructed a unipotent upper-triangular matrix $U$ such that $WU$ is symplectic or orthogonal with respect to $J$.  
\end{proof}


\begin{proposition}\label{minors for W} For any $1\leq m\leq k$ and any subset $S$ of cardinality $m$, we have the equality among the minors for the Wronskian matrix \er{Wrons}
$$
W_{S,\ul{m}} = W_{\iota(\bar S), \ul{k-m}}.
$$
\end{proposition}

\begin{proof} By Lemma \ref{GS}, the matrix $WU$ is symplectic or orthogonal. It is clear from the definition \er{inversion} that 
$$
\iota(\bar {\ul{m}}) = \ul{k-m}.
$$
Therefore \er{neat} gives that 
$$
(WU)_{S,\ul{m}} = (WU)_{\iota(\bar S),\ul{k-m}}.
$$
Since $U$ is a unipotent upper-triangular matrix, $WU$ is obtained from $W$ using column additions by previous columns. 
Therefore we see that  
$$
(WU)_{S,\ul{m}}=W_{S,\ul{m}}\quad\text{and}\quad (WU)_{\iota(\bar S),\ul{k-m}}=W_{\iota(\bar S),\ul{k-m}},
$$
since the column indices $\ul{m}$ and $\ul{k-m}$ are strings of numbers starting from 1. 
\end{proof}

Now we are ready for the following.

\begin{proof}[Proof of Proposition \ref{main}] Actually merely the case of \er{sym for U} when $i=1$ 
$$
e^{-\t U_1} = e^{-\t U_{k-1}}
$$
is sufficient for proving that $H$ is either symplectic or orthogonal. 
By \er{U1} and \er{Um}, we get 
\begin{equation}\label{firstlast}
R_{1, 1} = R_{\ul{k-1}, \ul{k-1}}.
\end{equation}
By \er{big S} and basic properties of determinants, we get 
\begin{gather*}
R_{1, 1} = \sum_{i, j=1}^k \overline{W_{i, 1}} H_{i,j} W_{j, 1},\\
R_{\ul{k-1}, \ul{k-1}} = \sum_{i, j=1}^k \ol{W_{\bar{\iota(i)}, \ul{k-1}}} H_{\bar{\iota(i)}, \bar{\iota(j)}} W_{\bar{\iota(j)}, \ul{k-1}},
\end{gather*}
where the second equation uses the fact that any $S\subset \{1, \cdots, k\}$ with $k-1$ elements can be written as the complement of one element $\iota(i)$ for some $1\leq i\leq k$. 

By Proposition \ref{minors for W} with $m=1$ and $S=\{i\}$ or $\{j\}$ for $1\leq i, j\leq k$, we have 
\begin{equation}\label{eq fns}
\overline{W_{i, 1}}=\ol{W_{\bar{\iota(i)}, \ul{k-1}}}, \quad 
W_{j, 1} = W_{\bar{\iota(j)}, \ul{k-1}}.
\end{equation}
Furthermore, by \er{Wrons}, \er{scale} and \er{new exp}, we have 
\begin{equation}\label{distinct}
W_{j, 1} = \nu_{j-1} = \chi_{j-1}\, z^{\t \mu_1 + \cdots + \t\mu_{j-1} - \a_1}.
\end{equation}
where 
$\t\mu_i = \t\gamma_i + 1> 0$. Therefore the $W_{j, 1}$ have strictly increasing exponents. Similarly this holds for the $\ol{W_{i, 1}}$. Hence the equalities \er{firstlast} and \er{eq fns} mean that the coefficients are equal, that is, 
$$
H_{i, j} = H_{\bar{\iota(i)}, \bar{\iota(j)}}, \quad \forall 1\leq i, j\leq k.
$$
Then \er{minor} in Proposition \ref{prop-minors} implies that $H$ is symplectic when $k$ is even or orthogonal when $k$ is odd. 

Furthermore along the same lines, it is easy to see that as long as $H$ is symplectic or orthogonal, the requirement \er{sym for U} is satisfied for all $i$, that is
$$
e^{-\t U_i} = R_{\ul{i}, \ul{i}} = R_{\ul{k-i}, \ul{k-i}} = e^{-\t U_{k-i}},
$$
by determinant expansion, Propositions \ref{minors for W} and \er{neat} of Proposition \ref{prop-minors}. 
\end{proof}

As mentioned in Section \ref{revisit}, \cite{LWY} shows that the Hermitian $H$ is positive definite and hence has the Cholesky decomposition \er{UL}. Now we show that the corresponding lower-triangular matrix $B$ must also be symplectic or orthogonal according to $H$. 

\begin{proposition}\label{make them} For a Hermitian positive-definite symplectic or orthogonal matrix $H$, consider its Cholesky decomposition $H = B^\dag B$ where $B$ is lower-triangular. 
Let $B=\Lambda C$ be the decomposition of $B$ into its diagonal and unipotent parts with $\Lambda=\on{diag}(\lambda_0, \cdots, \lambda_{k-1})$ where $\lambda_i>0$ and $C$ unipotent. Then $C, \Lambda$ and hence $B$ are all symplectic or orthogonal. 
\end{proposition}

\begin{proof} Using $H= B^\dag B$ in $H^t J H = J$, we get 
$
B^t \bar B J \bar B^t B = J,
$
which means 
\begin{equation}\label{spell}
\bar B J \bar B^t = (B^{-1})^t J B^{-1}.
\end{equation}
Now $B$ is lower-triangular, and so are $\bar B$ and $B^{-1}$. Also their transposes are upper-triangular. It is easy to see that with respect to the secondary diagonal, the left hand side is lower-triangular while the right is upper-triangular. For the equality to hold, 
we need that both sides are secondary diagonal. 

Plug $B=\Lambda C$ in the right hand side of \er{spell} to get $ \Lambda^{-1} (C^{-1})^t J C^{-1} \Lambda^{-1}$. Since $\Lambda$ is diagonal, $(C^{-1})^t J C^{-1}$ is secondary diagonal and is equal to $J$ since $C$ is unipotent. Therefore $C^{-1}$ and hence $C$ are symplectic or orthogonal. 

Now \er{spell} becomes 
$$
\Lambda J \Lambda = \Lambda^{-1} J \Lambda^{-1}.
$$
This says that $\Lambda^2$ is symplectic or orthogonal and hence $\Lambda$ is. Concretely this means that $\lambda_i \lambda_{k-1-i} = 1$ for $0\leq i\leq k-1$.
\end{proof}

\section{The subgroups parametrizing the solutions}\label{further}

The final piece of information is to study the unipotent subgroup $N$ of lower-triangular matrices in $Sp(2n, \bc)$ or $SO(2n+1, \bc)$. In this section we will first set up a coordinate system on $N$. 
Corresponding to the conditions \er{c=0} and \er{A cond} in the $A$ case, there are some further restrictions on $N$ for the solutions
to be well-defined on $\bc^*$
depending on the $\gamma_i$, 
and we will pin down the restrictions. 

The unipotent subgroup $N$ is simply-connected and diffeomorphic to a complex vector space, in particular to its Lie algebra $\fn$. One can use the so-called coordinates of the second type, and we refer the reader to \cite[p. 76, Cor. 1.126, Thm. 6.46]{Knapp} for more details. Here we would like to use a slight variation to the coordinates of the second type, which is more convenient for our purposes. 

A unipotent lower-triangular matrix, when considered as an element in $SL(k, \bc)$, can be written as 
\begin{equation}\label{concrete N}
C = \begin{pmatrix}
1 & & & & \\
c_{10} & 1 & & & \\
c_{20} & c_{21} & 1 & & \\
\vdots & \vdots & \ddots & \ddots & \\
c_{k-1, 0} & c_{k-1, 1} & \cdots & c_{k-1, k-2} & 1
\end{pmatrix}.
\end{equation}
For $C$ to be symplectic or orthogonal, it satisfies the condition $C^t J_k C = J_k$.

When $k=2n$ is even, $C$ is symplectic and we can choose $c_{ij}$ for $0\leq j<i\leq 2n-1-j, \ 0\leq j\leq n-1$ to be our coordinates on $N$. The dimension is $n^2$. The other $c$'s in $C$ can be solved in term of these by the above condition. 
Similarly when $k=2n+1$ is odd, $C$ is orthogonal and we can still choose $c_{ij}$ for $0\leq j<i\leq 2n-1-j, \ 0\leq j\leq n-1$ to be our coordinates, while the other $c$'s can be solved in term of these. 


In this section, we will formulate our results in a Lie-theoretical way. We first introduce a bit more terminology from Lie theory and refer the reader to, for example, \cite{Knapp, FH}. Let $\fg$ be a complex simple Lie algebra of rank $n$. Let $\tau_1, \dots, \tau_n$ be the 
simple roots, and let $\Delta^+$ denote the set of positive roots of $\fg$. Let $\fg = \fh \oplus \bigoplus_{\tau\in \Delta^+} (\fg_\tau \oplus \fg_{-\tau})$ be the root space decomposition, where $\fh$ is the Cartan subalgebra and $\fg_\tau$ is the root space corresponding to $\tau$ generated by a root vector $e_\tau$. Let $\fn = \bigoplus_{\tau\in \Delta^+}  \fg_{-\tau}$ be the negative nilpotent Lie algebra. Let $G$ be a Lie group corresponding to $\fg$, and $N$ the subgroup corresponding to $\fn$. 

The above coordinates on $N$ in $Sp(2n, \bc)$ and $SO(2n+1, \bc)$ correspond to the roots in the sense that $\frac{\p}{\p c_{ij}} \big|_{c=0} C$ in \er{concrete N} is a root vector in the Lie algebra. For completeness, we list the correspondences. The Cartan subalgebra $\fh$ is chosen to consists of diagonal matrices, and $L_i$ denotes the linear function on $\fh$ taking the $i$th diagonal entry. 

The positive roots for $C_n$ are 
\begin{equation*}\label{C-roots}
\begin{cases}
\tau_i + \cdots + \tau_{j-1}=L_i - L_j, & \text{for }1\leq i < j \leq n,\\
(\tau_i + \cdots \tau_{n-1}) + (\tau_j + \cdots + \tau_n)=L_i + L_j, & \text{for }1\leq i\leq j\leq n.
\end{cases}
\end{equation*}
There are totally $n^2$ positive roots. 
Furthermore, the coordinates $c_{ij}$ in the $C_n$ case correspond to the negative roots as follows. The $c_{ij}$ for $0\leq j<i \leq n-1$ corresponds to the negative of $L_{j+1} - L_{i+1}$, and the $c_{ij}$ for $0\leq j < n\leq i\leq 2n-1-j$ corresponds to the negative of $L_{j+1} + L_{2n-i}$. 

The positive roots for $B_n$ are 
\begin{equation*}\label{B-roots}
\begin{cases}
\tau_i + \cdots + \tau_{j}=L_i - L_{j+1}, & \text{for }1\leq i \leq j \leq n,\\
(\tau_i + \cdots \tau_{n}) + (\tau_j + \cdots + \tau_n)=L_i + L_j, & \text{for }1\leq i< j\leq n.
\end{cases}
\end{equation*}
We note that $L_{n+1}=0$ for $B_n$. 
There are again $n^2$ positive roots. 
Furthermore, the coordinates $c_{ij}$ in the $B_n$ case correspond to the negative roots as follows. The $c_{ij}$ for $0\leq j<i \leq n$ corresponds to the negative of $L_{j+1} - L_{i+1}$, and the $c_{ij}$ for $0\leq j < n < i\leq 2n-1-j$ corresponds to the negative of $L_{j+1} + L_{2n+1-i}$.

Using Proposition \ref{make them}, the formula \er{new H} for the solution to the system \er{toda3} with \er{sym for gamma} and \er{sym for U} becomes 
\begin{equation}\label{con form}
e^{-\t U_1} = \bnu^\dag B^\dag B \bnu =  | \Lambda C \bnu |^2 = \sum_{i=0}^{k-1} \lambda_i^2 \Big| \nu_i + \sum_{j=0}^{i-1} c_{ij} \nu_j \Big|^2,
\end{equation}
where the diagonal matrix is $\Lambda=\on{diag}(\lambda_0, \cdots, \lambda_{k-1})$ satisfying $\lambda_i \lambda_{k-1-i} = 1$, and the unipotent matrix $C$ has coordinates as above. 

We now study the conditions such that \er{con form} is well-defined on $\bc^*$. Let $\Gamma=(\gamma_1, \cdots, \gamma_n)$ be the $n$-tuple with entries $\gamma_i>-1$ from \er{toda}. We define a subalgebra $\fn_\Gamma$ and a subgroup $N_\Gamma$ as follows. 

\begin{definition}\label{in roots} For $\tau\in \Delta^+$, write $\tau = \sum_{i=1}^n m_i \tau_i$ in terms of the simple roots. Define the number $\tau(\Gamma) = \sum_{i=1}^n m_i \gamma_i$ where we replace $\tau_i$ by $\gamma_i$. 

Define the subset $\Delta_\Gamma$ of $\Delta^+$ as 
$
\Delta_\Gamma = \{ \tau\in \Delta^+ \,|\, \tau(\Gamma)\in \bz\}.
$
Note that $\Delta_\Gamma$ is closed under addition, that is, if $\a, \beta\in \Delta_\Gamma$ and $\a + \beta\in \Delta^+$, then $\a + \beta \in \Delta_\Gamma$. 

Define the Lie subalgebra $\fn_\Gamma$ of $\fn$ as 
$
\fn_\Gamma = \bigoplus_{\tau \in \Delta_\Gamma} \fg_{-\tau},
$
and the subgroup $N_\Gamma$ of $N$ as the corresponding Lie subgroup, 
$
N_\Gamma = \exp(\fn_\Gamma).
$
\end{definition}

Concretely for the $C_n$ and $B_n$ cases, $N_\Gamma$ consists of matrices \er{concrete N} where we let a coordinate $c_{ij} = 0$ if the corresponding root does not belong to $\Delta_\Gamma$. 

Inspired by a discussion with Ming Xu on conditions \er{c=0} and \er{A cond} in the $A$ case, we give the subgroup $N_\Gamma$ the following more intrinsic interpretation. There are grading elements $E_j\in \fh$ in the Cartan subalgebra such that $\tau_i (E_j) = \delta_{ij}$, where $\tau_i\in \fh^*$ is regarded as a linear function on $\fh$. Consider the following element in the Cartan subgroup 
\begin{equation}\label{cent}
g_\Gamma := \exp\Big(2\pi i\sum_{j=1}^n \gamma_j E_j\Big) = \exp\Big(2\pi i\sum_{j=1}^n \alpha_j H_j\Big) ,
\end{equation}
where $H_j = [e_{\tau_j}, e_{-\tau_j}]$ and $\tau_j(H_j)= 2$. 
The number $\tau(\Gamma)$ in Definition \ref{in roots} is seen to be $\tau(\Gamma) = \tau(\sum_j \gamma_j E_j)$, and so $\on{Ad}_{g_\Gamma} e_\tau = \exp(2\pi i \tau(\Gamma)) e_\tau$. Hence we see that 
\begin{equation}\label{intrinsic} 
\fn_\Gamma = \fn^{\on{Ad}_{g_\Gamma}}, \quad 
N_\Gamma = N^{\on{Ad}_{g_\Gamma}}
\end{equation}
are the fixed point sets of the adjoint actions by $g_\Gamma$. 

\begin{proposition}\label{NGm} For the $e^{-\t U_1}$ in \er{con form} to be well-defined on $\bc^*$, the matrix $C$ in \er{concrete N} must belong to the subgroup $N_\Gamma$. 
\end{proposition}

\begin{proof}
By \er{c=0} in Theorem \ref{A case}, we see that the coefficient $c_{ij} = 0$ if the $\t \gamma_{j+1} + \cdots + \t\gamma_i\notin \bz$. It can be seen that such information is exactly encoded in the root structure of the subgroup $N_\Gamma$. 

We can also argue using \er{intrinsic} as follows. For \er{con form} to be well-defined on $\bc^*$, we need it to be invariant under the change of $z\mapsto e^{-2\pi i} z$. Using \er{exponents}, \er{C_n alp}, \er{B_n alp}, and \er{cent}, it can be seen that under this change, $\bnu \mapsto g_\Gamma \bnu$. 
Therefore the solution \er{con form} becomes 
$
\bnu^\dag g_\Gamma^\dag B^\dag B g_\Gamma \bnu.
$

Arguing as in \er{distinct}, we have 
$$B^\dag B = g_\Gamma^\dag B^\dag B g_\Gamma =  (g_\Gamma^\dag B g_\Gamma)^\dag (g_\Gamma^\dag B g_\Gamma).$$ 
Since $g_\Gamma^\dag B g_\Gamma$ is still lower-triangular with positive diagonal entries, by the uniqueness of the Cholesky decomposition \cite{GvL}, we see that 
$$
g_\Gamma^\dag B g_\Gamma = B. 
$$
Using $B=\Lambda C$, we see that $g_\Gamma^{-1} C g_\Gamma =  C$ and $C\in N_\Gamma$ by \er{intrinsic}. 
\end{proof}

\appendix
\section{The example of $C_3$}
In the two appendices, we work out the examples of $C_3$ and $B_2$ Toda systems to demonstrate our results. 

We first consider the $C_3$ case where we set out to solve the system \er{toda2} where the Cartan matrix is \er{Cartan-C} for $n=3$. 

The $\alpha$ in \er{def alpha} are defined in terms of the $\gamma$ by  
$$
\begin{pmatrix}
\alpha_1 \\ \alpha_2 \\ \alpha_3
\end{pmatrix} 
=
\begin{pmatrix}
1&1&\frac{1}{2}\\
1&2&1\\
1&2&\frac{3}{2}
\end{pmatrix}
\begin{pmatrix}
\gamma_1 \\ \gamma_2 \\ \gamma_3
\end{pmatrix},
$$
where the matrix is the inverse Cartan matrix of $C_3$. 
With $\mu_i = \gamma_i + 1$ for $1\leq i\leq 3$ and by Remark \ref{repeation} and \er{iteint}, the $\bnu$ has the following shape 
\begin{multline}
\bnu = \frac{1}{z^{\gamma_1+\gamma_2+\frac{\gamma_3}{2}}}\Big(1, \frac{z^{\mu_1}}{\mu_1}, \frac{z^{\mu_1 + \mu_2}}{\mu_2(\mu_1 + \mu_2)}, \frac{z^{\mu_1 + \mu_2 + \mu_3}}{\mu_3(\mu_2 + \mu_3)(\mu_1+ \mu_2 + \mu_3)},\\ 
\frac{z^{\mu_1 + 2\mu_2 + \mu_3}}{\mu_2(\mu_2 + \mu_3)(2\mu_2 + \mu_3)(\mu_1+ 2\mu_2 + \mu_3)}, \\
\frac{z^{2\mu_1 + 2\mu_2 + \mu_3}}{\mu_1(\mu_2 + \mu_1)(\mu_1+ \mu_2 + \mu_3)(\mu_1+ 2\mu_2 + \mu_3)(2\mu_1+ 2\mu_2 + \mu_3)}\Big)^t.
\end{multline}

By Proposition \ref{make them}, the $\Lambda$ and $C$ have the following shapes:
\begin{gather*}
\Lambda = \on{diag}(\lambda_1, \lambda_2, \lambda_3, \lambda_3^{-1}, \lambda_2^{-1}, \lambda_1^{-1}),\quad \lambda_i>0,\\
C = \begin{pmatrix}
1 & & & & & \\
c_{10} & 1 & & & & \\
c_{20} & c_{21} & 1 & & & \\
c_{30} & c_{31} & c_{32} & 1 & & \\
c_{40} & c_{41} & c_{42} & c_{43} & 1 & \\
c_{50} & c_{51} & c_{52} & c_{53} & c_{54}& 1
\end{pmatrix}
\end{gather*}
where the $c_{51}, c_{52}, c_{42}, c_{53}, c_{43}, c_{54}$ are easily solved in terms of the others by the condition that 
$
C^t J_6 C = J_6. 
$
The results are 
\begin{gather*}
c_{{51}}=c_{{10}}c_{{41}}-c_{{20}}c_{{31}}+c_{{30}}c_{{21}}-c_{{40}},\\
c_{{42}}=c_{{21}}c_{{32}}-c_{{31}},\quad c_{{52}}=c_{{10}}c_{{21}}c_{{32}}-c_{{10}}c_{{31}}-c_{{20}
}c_{{32}}+c_{{30}},\\
c_{{43}}=c_{{21}},\quad c_{{53}}=c_{{10}}c_{{21}}-c_{{20}},\quad c_{{54}}=c_{{10}}.
\end{gather*}

The element in \er{cent} is 
\begin{align*}
g_\Gamma &= \exp\big(2\pi i \on{diag}(\a_1, \a_2-\a_1, \a_3-\a_2, \a_2-\a_3,  \a_1-\a_2, -\a_1)\big)
\end{align*}
and the $C$ is required by Proposition \ref{NGm} to satisfy that 
$$
C g_\Gamma = g_\Gamma C. 
$$

This restricts some of the free parameters to zero depending on the integrability of the $\gamma$. 
More concretely, we have the following correspondence between the coordinates and the roots
\medskip
\begin{center}
\begin{tabular}{cccccccc}
\hline
coord's & $c_{10}$ & $c_{21}$ & $c_{32}$ & $c_{20}$ & $c_{31}$ & $c_{30}$ & $c_{41}$  \\
roots & $\tau_1$ & $\tau_2$ & $\tau_3$ & $\tau_1 + \tau_2$ & $\tau_2 + \tau_3$ & $\tau_1 + \tau_2 + \tau_3$ & $2\tau_2 + \tau_3$  \\
\hline
& & & & & & $c_{40}$ & $c_{50}$  \\
& & & & & & $\tau_1 + 2\tau_2 + \tau_3$ & $2\tau_1 + 2\tau_2 + \tau_3$ \\
\cline{7-8}
\end{tabular}
\end{center}
\medskip
A coordinate must be zero if for the corresponding root $\tau\in \Delta^+$, its value $\tau(\Gamma)$ with the $\tau_i$ replaced by $\gamma_i$ is not an integer. 

For example, when $\gamma_1=-0.5, \gamma_2=0.25, \gamma_3=1$, the roots with integral values are $\tau_3$ and $\tau_1 + 2\tau_2 + \tau_3$ and the nonzero coordinates are $c_{32}$ and $c_{40}$. 

\section{The example of $B_2$}
For completeness, we also show the solutions to the $B_2$ Toda system \er{toda2} as the smallest example of type $B$. Of course $B_2$ is isomorphic to $C_2$, and our result can be compared with those in \cite{ALW} where it was the $C_2$ case that was treated. 

In this case, the $\alpha$ are defined in terms of the $\gamma$ by  
$$
\begin{pmatrix}
\alpha_1 \\ \alpha_2 
\end{pmatrix} 
=
\begin{pmatrix}
1&1\\
\frac{1}{2} & 1
\end{pmatrix}
\begin{pmatrix}
\gamma_1 \\ \gamma_2 
\end{pmatrix},
$$
where the matrix is the inverse Cartan matrix of $B_2$. 
With $\mu_i = \gamma_i + 1$ for $1\leq i\leq 2$ and by Remark \ref{repeation} and \er{iteint}, the $\bnu$ has the following shape
\begin{multline}
\bnu = \frac{1}{z^{\gamma_1 + \gamma_2}}\Big(1, \frac{z^{\mu_1}}{\mu_1}, \frac{z^{\mu_1 + \mu_2}}{\mu_2(\mu_1 + \mu_2)}, \frac{z^{\mu_1 + 2\mu_2 }}{2\mu_2^2(\mu_1 + 2\mu_2)},
\frac{z^{2\mu_1 + 2\mu_2}}{2\mu_1(\mu_1 + \mu_2)^2(\mu_1 + 2\mu_2)}\Big)^t.
\end{multline}

By Proposition \ref{make them}, our $\Lambda$ and $C$ have the following shapes:
\begin{gather*}
\Lambda = \on{diag}(\lambda_1, \lambda_2, 1, \lambda_2^{-1}, \lambda_1^{-1}),\quad \lambda_i>0,\\
C = \begin{pmatrix}
1 & & & & \\
c_{10} & 1 & & & \\
c_{20} & c_{21} & 1 & & \\
c_{30} & c_{31} & c_{32} & 1 & \\
c_{40} & c_{41} & c_{42} & c_{43} & 1\\
\end{pmatrix}
\end{gather*}
where the $c_{40}, c_{31}, c_{41}, c_{32}, c_{42}, c_{43}$ are easily solved in terms of the others by the condition that 
$
N^t J_5 N = J_5.
$
The results are 
\begin{gather*}
c_{{40}}=c_{{10}}c_{{30}}-\frac{1}{2}c_{{20}}^{2},\\
c_{{31}}=\frac{1}{2}c_{{21}}^{2},\quad c_{{41}}=\frac{1}{2} c_{{10}}c_{{2
1}}^{2}-c_{{20}}c_{{21}}+c_{{30}}\\
c_{{32}}=c_{{21}},\quad c_{{42}}=c_{{10}}c_{{21}}-c_
{{20}},\quad c_{{43}}=c_{{10}}
\end{gather*}

The element in \er{cent} is 
\begin{align*}
g_\Gamma &= \exp\big(2\pi i \on{diag}(\a_1, 2\a_2-\a_1, 0, \a_1-2\a_2, -\a_1)\big)
\\ &=  \exp\big(2\pi i \on{diag}(\gamma_1+\gamma_2, \gamma_2, 0, -\gamma_2, -\gamma_1-\gamma_2)\big),
\end{align*}
and the $C$ is required to satisfy that 
$$
C g_\Gamma = g_\Gamma C.
$$

This restricts some of the free parameters to zero depending on the if the integrability of the $\gamma$. 
More concretely, we have the following correspondence between the coordinates and the roots
\medskip
\begin{center}
\begin{tabular}{ccccc}
\hline
coord's & $c_{10}$ & $c_{21}$ & $c_{20}$ & $c_{30}$  \\
roots & $\tau_1$ & $\tau_2$ & $\tau_1 + \tau_2$ & $\tau_1 + 2\tau_2$ \\
\hline
\end{tabular}
\end{center}
\medskip
A coordinate must be zero if for the corresponding roots $\tau\in \Delta^+$, its value $\tau(\Gamma)$ with the $\tau_i$ replaced by $\gamma_i$ is not an integer. 

For example, when $\gamma_1=-0.5$ and $\gamma_2=0.25$, the only root with integral value is $\tau_1 + 2\tau_2$ and the nonzero coordinate is $c_{30}$. 

These give the formula \er{con form}, and the solution $U_1$ to the $B_2$ Toda system is $U_1 = \t U_1 - \ln 2$ by \er{BU}.



\bigskip

\begin{bibdiv}
\begin{biblist}

\bib{ALW}{article}{
   author={Ao, Weiwei},
   author={Lin, Chang-Shou},
   author={Wei, Juncheng},
   title={On Non-topological Solutions of the $A_2$ and $B_2$ Chern-Simons System},
   status = {to appear in Memoirs of the A.M.S.},
   year={2013}
}

\bib{ALW2}{article}{
   author={Ao, Weiwei},
   author={Lin, Chang-Shou},
   author={Wei, Juncheng},
   title={On Toda system with Cartan matrix $G_2$},
   journal={Proc. Amer. Math. Soc.},
   volume={143},
   date={2015},
   number={8},
   pages={3525--3536},
}

\bib{Feher}{article}{
   author={Balog, J.},
   author={Feh{\'e}r, L.},
   author={O'Raifeartaigh, L.},
   author={Forg{\'a}cs, P.},
   author={Wipf, A.},
   title={Toda theory and $\scr W$-algebra from a gauged WZNW point of view},
   journal={Ann. Physics},
   volume={203},
   date={1990},
   number={1},
   pages={76--136},
   issn={0003-4916},
}


\bib{BM}{article}{
   author={Brezis, H.},
   author={Merle, F.},
   title={Uniform estimates and blow-up behavior for solutions of $-\Delta u=V(x)e^u$ in two dimensions},
   journal={Comm. Partial Differential Equations},
   volume={16},
   date={1991},
   number={8-9},
   pages={1223--1253},
   issn={0360-5302},
}

\bib{CI}{article}{
   author={Chae, Dongho},
   author={Imanuvilov, Oleg Yu.},
   title={The existence of non-topological multivortex solutions in the
   relativistic self-dual Chern-Simons theory},
   journal={Comm. Math. Phys.},
   volume={215},
   date={2000},
   number={1},
   pages={119--142},
   issn={0010-3616},
}

\bib{CL}{article}{
   author={Chen, Wen Xiong},
   author={Li, Congming},
   title={Classification of solutions of some nonlinear elliptic equations},
   journal={Duke Math. J.},
   volume={63},
   date={1991},
   number={3},
   pages={615--622},
   issn={0012-7094},
}

\bib{FH}{book}{
   author={Fulton, William},
   author={Harris, Joe},
   title={Representation theory},
   series={Graduate Texts in Mathematics},
   volume={129},
   note={A first course;
   Readings in Mathematics},
   publisher={Springer-Verlag},
   place={New York},
   date={1991},
   pages={xvi+551},
   isbn={0-387-97527-6},
   isbn={0-387-97495-4},
}

\bib{GNN}{article}{
   author={Gidas, B.},
   author={Ni, Wei Ming},
   author={Nirenberg, L.},
   title={Symmetry of positive solutions of nonlinear elliptic equations in
   ${\bf R}^{n}$},
   conference={
      title={Mathematical analysis and applications, Part A},
   },
   book={
      series={Adv. in Math. Suppl. Stud.},
      volume={7},
      publisher={Academic Press, New York-London},
   },
   date={1981},
   pages={369--402},
}

\bib{GvL}{book}{
   author={Golub, Gene H.},
   author={Van Loan, Charles F.},
   title={Matrix computations},
   series={Johns Hopkins Studies in the Mathematical Sciences},
   edition={3},
   publisher={Johns Hopkins University Press, Baltimore, MD},
   date={1996},
   pages={xxx+698},
}

\bib{Helgason}{book}{
   author={Helgason, Sigurdur},
   title={Differential geometry, Lie groups, and symmetric spaces},
   series={Pure and Applied Mathematics},
   volume={80},
   publisher={Academic Press, Inc. 
   New York-London},
   date={1978},
   pages={xv+628},
   isbn={0-12-338460-5},
}

\bib{JW}{article}{
   author={Jost, J{\"u}rgen},
   author={Wang, Guofang},
   title={Classification of solutions of a Toda system in ${\Bbb R}^2$},
   journal={Int. Math. Res. Not.},
   date={2002},
   number={6},
   pages={277--290},
   issn={1073-7928},
}

\bib{Knapp}{book}{
   author={Knapp, Anthony W.},
   title={Lie groups beyond an introduction},
   series={Progress in Mathematics},
   volume={140},
   edition={2},
   publisher={Birkh\"auser Boston, Inc., Boston, MA},
   date={2002},
   pages={xviii+812},
   isbn={0-8176-4259-5},
}

\bib{L}{article}{
   author={Leznov, A. N.},
   title={On complete integrability of a nonlinear system of partial
   differential equations in two-dimensional space},
   journal={Teoret. Mat. Fiz.},
   volume={42},
   date={1980},
   number={3},
   pages={343--349},
   issn={0564-6162},
}

\bib{LS}{article}{
   author={Leznov, A. N.},
   author={Saveliev, M. V.},
   title={Representation of zero curvature for the system of nonlinear
   partial differential equations $x_{\alpha ,z\bar z}={\rm
   exp}(kx)_{\alpha }$ and its integrability},
   journal={Lett. Math. Phys.},
   volume={3},
   date={1979},
   number={6},
   pages={489--494},
   issn={0377-9017},
}


\bib{LS-book}{book}{
   author={Leznov, A. N.},
   author={Saveliev, M. V.},
   title={Group-theoretical methods for integration of nonlinear dynamical
   systems},
   series={Progress in Physics},
   volume={15},
   note={Translated and revised from the Russian;
   Translated by D. A. Leuites},
   publisher={Birkh\"auser Verlag},
   place={Basel},
   date={1992},
   pages={xviii+290},
   isbn={3-7643-2615-8},
}


\bib{LWY}{article}{
   author={Lin, Chang-Shou},
   author={Wei, Juncheng},
   author={Ye, Dong},
   title={Classification and nondegeneracy of $SU(n+1)$ Toda system with
   singular sources},
   journal={Invent. Math.},
   volume={190},
   date={2012},
   number={1},
   pages={169--207},
   issn={0020-9910},
}

\bib{N1}{article}{
   author={Nie, Zhaohu},
   title={Solving Toda field theories and related algebraic and differential
   properties},
   journal={J. Geom. Phys.},
   volume={62},
   date={2012},
   number={12},
   pages={2424--2442},
   issn={0393-0440},
}


\bib{N2}{article}{
   author={Nie, Zhaohu},
   title={On characteristic integrals of Toda field theories},
   journal={J. Nonlinear Math. Phys.},
   volume={21},
   date={2014},
   number={1},
   pages={120--131},
   issn={1402-9251},
}

\bib{Yang}{book}{
   author={Yang, Yisong},
   title={Solitons in field theory and nonlinear analysis},
   series={Springer Monographs in Mathematics},
   publisher={Springer-Verlag, New York},
   date={2001},
   pages={xxiv+553},
   isbn={0-387-95242-X},
}

		
\end{biblist}
\end{bibdiv}
\end{document}